\documentclass[twoside]{article}
\usepackage[cp1251]{inputenc}
\usepackage[T2A]{fontenc}
\usepackage{array}
\usepackage{amssymb}
\usepackage{amsmath}
\usepackage{amsthm}
\usepackage{latexsym}
\usepackage{indentfirst}
\usepackage{bm}
\usepackage{enumerate}
\usepackage[dvips]{graphicx}
\usepackage{epsf}
\usepackage{wrapfig}
\usepackage{euscript}
\usepackage{indentfirst}
\usepackage[english,russian]{babel}
\usepackage{textcomp}
\newtheorem{theorem}{Theorem}
\newtheorem{lemma}{Lemma}
\newtheorem{definition}{Definition}

\begin{document}
{\selectlanguage{english}
\binoppenalty = 10000 %
\relpenalty   = 10000 %

\pagestyle{headings} \makeatletter
\renewcommand{\@evenhead}{\raisebox{0pt}[\headheight][0pt]{\vbox{\hbox to\textwidth{\thepage\hfill \strut {\small Grigory. K. Olkhovikov}}\hrule}}}
\renewcommand{\@oddhead}{\raisebox{0pt}[\headheight][0pt]{\vbox{\hbox to\textwidth{{Characterization of predicate intuitionistic formulas}\hfill \strut\thepage}\hrule}}}
\makeatother

\title{Intuitionistic predicate logic of constant domains does not have Beth property}
\author{Grigory K. Olkhovikov\\ Department of Ontology and Cognition Theory\\ Ural Federal University\\
Fulbright Visiting Scholar at the Philosophy Dept,\\Stanford
University \\Bldg 90, Stanford, CA, USA}
\date{}
\maketitle
\begin{quote}
{\bf Abstract.} Drawing on the previous work \cite{MOU2012} on
interpolation failure, we show that Beth's definability theorem
does not hold for intuitionistic predicate logic of constant
domains without identity.
\end{quote}

It is known that intuitionistic predicate logic of constant
domains without identity does not have Craig interpolation
property. Given that this property is normally used to derive,
among other things, Beth's definability theorem, the failure of
interpolation suggests a hypothesis that for the mentioned logic
Beth's definability property fails, too. In the paper expounding
the non-interpolation result its authors mention this hypothesis:
``It is not known at this moment whether \textbf{CD} has the Beth
property, although this does not look plausible'' \cite[p.
3]{MOU2012}.

In the present paper we show that this hypothesis is in fact true.
The counterexample we give for Beth's definability property of
\textbf{CD} is in close and obvious connection with the
counterexample for interpolation given in \cite{MOU2012}, and we
use similar methods to prove that it is in fact a counterexample.
However, the models used in the counterexample are somewhat
different and more complicated than the ones used to disprove
interpolation. This necessitates introduction of some new
technical notions that were not required for the interpolation
failure proof and discussion of some of their properties.

\begin{definition} A theory $T$ in language $L$ \emph{implicitly defines}
$P \in L$ iff for any model $\mathcal{M}$ of $T$ there is no model
$\mathcal{N}$ of $T$ that would differ from $\mathcal{M}$ only in
the extension of $P$.  A theory $T$ in language $L$
\emph{explicitly defines} $P \in L$ iff there is a formula
$\Theta(\vec{x}) \in L \smallsetminus \{ P \}$, such that $T
\models \forall\vec{x}(\Theta(\vec{x}) \leftrightarrow
P(\vec{x}))$. A logic has \emph{Beth's definability property} iff
for any theory $T$ in any language $L$ of this logic and any $P
\in L$, if $T$ defines $P$ implicitly, then $T$ defines $P$
explicitly.
\end{definition}

 We assume the results proved and the notions defined in
 \cite{MOU2012}
and notational conventions used there.

Consider function
\begin{align*}
\gamma(n) =
\begin{cases}
3n + 1, \text{ if } n \in 3\mathbb{N} \cup (3\mathbb{N} + 2);\\
n\text{ otherwise}.
\end{cases}
\end{align*}

and relation $R(x, y) \Leftrightarrow (x = \gamma(y) \vee y =
\gamma(x))$. If $X \subseteq \mathbb{N}$ then we call it
\emph{closed} iff it is closed with respect to $R$, and we take
the \emph{closure of }$X$, $Cl(X)$ to be the least $Y$ such that
$X \subseteq Y \subseteq \mathbb{N}$ and $Y$ is closed. For  $X
\subseteq \mathbb{N}$ we denote $Cl(X) \smallsetminus X$ by
$Cl^-(X)$.

It is easy to see that our closures have a very special structure:
every closed set is representable as a sum of singletons and pairs
which are closures of its points. So the following notion of
\emph{the companion of a natural} $n$ (which we denote by
$\tilde{n}$) turns out to be useful:
\begin{align*}
\tilde{x} =
\begin{cases}
y,\text{ if }x \neq y \in Cl(\{ x \});\\
x\text{ otherwise}.
\end{cases}
\end{align*}
We denote the set of natural numbers that are their own companions
by $\mathbb{N}_0$. Thus $\mathbb{N}_0 = \{ n \mid \gamma(n) = n
\}$. The following lemma states some obvious properties of the
notions defined above that we will need in what follows:
\begin{lemma}
Let $n \in\mathbb{N}$, $X \subseteq \mathbb{N}$. Then the
following statements are true:
\begin{align}
&\mathbb{N}_0 = \mathbb{N} \smallsetminus Cl(3 \mathbb{N} \cup (3
\mathbb{N} + 2));\label{En0}\\
&\mathbb{N}_0\text{ is infinite;}\label{En0inf}\\
&Cl(X)\text{ is finite } \Leftrightarrow Cl^-(X)\text{ is finite } \Leftrightarrow X\text{ is finite;}\label{Eclofin}\\
&n = \tilde{\tilde{n}};\label{Etilde}\\
&n \in 3\mathbb{N} \cup 3\mathbb{N} + 2 \Leftrightarrow \tilde{n}
\in
Cl^-(3\mathbb{N} \cup 3\mathbb{N} + 2);\label{Eiff1}\\
&\tilde{n} \in 3\mathbb{N} \cup 3\mathbb{N} + 2 \Leftrightarrow n
\in
Cl^-(3\mathbb{N} \cup 3\mathbb{N} + 2);\label{Eiff2}\\
&Cl^-(X) \subseteq 3\mathbb{N} + 1\label{Eclo-}\\
&X = Cl(X) \Leftrightarrow \mathbb{N} \smallsetminus X = Cl(
\mathbb{N} \smallsetminus X)\label{Ecloneg}
\end{align}
\end{lemma}
\begin{proof}
Most of the statements are in fact self-evident given the above
definitions. As an example we sketch proofs of the following ones:

\eqref{En0inf} Let $n_0 = 1$, $n_{i + 1} = 3n_i + 1$. An easy
induction shows that $n_i \neq n_j$ for $i \neq j$ and that all
$n_i$ are in $\mathbb{N}_0$.

\eqref{Ecloneg} Let $X$ be a closed subset of $\mathbb{N}$ and let
$n \notin X$. If $\tilde{n} \in X$, then, by closure of $X$ and
\eqref{Etilde}, $n = \tilde{\tilde{n}} \in X$ which contradicts
the choice of $n$. Therefore, $\mathbb{N} \smallsetminus X$ is
closed.
\end{proof}

We consider the following quasi-partitions:

\[
\mathbf{v} = (\mathbf{v}_1, \mathbf{v}_2, \mathbf{v}_3) = (
\mathbb{N} \smallsetminus Cl(3 \mathbb{N} + 2), Cl^-(3 \mathbb{N}
+ 2),(3 \mathbb{N} + 2) )
\]
\[
\mathbf{u} = (\mathbf{u}_1, \mathbf{u}_2, \mathbf{u}_3) =
(\mathbf{v}_1, \emptyset, \mathbf{v}_2 \cup \mathbf{v}_3)
\]

It is clear that these are in fact quasi-partitions.

 Moreover, consider set
\[
U = \{(A,B,C) \mid  \mathbf{v} \trianglelefteq (A,B,C), A\text{ is
closed, }B\subseteq \mathbf{v}_2\}
\]

It follows from \eqref{Ecloneg} that $\mathbf{v}_1$ is closed.
Therefore, $U$ is non-empty, for example, $\mathbf{v} \in U$.

The following lemma states some obvious properties of the defined
quasi-partitions that we will need in what follows:
\begin{lemma}
Let $w = (w_1, w_2, w_3) \in W_2$.Then the following statements
are true:
\begin{align}
&\mathbf{u} \trianglelefteq \mathbf{v};\label{Etriangle}\\
&w_2 \subseteq \mathbf{v}_2;\label{Ev2}\\
&Cl(3\mathbb{N}) \cup \mathbb{N}_0 \subseteq \mathbf{v}_1 =  \mathbf{u}_1 \subseteq w_1\label{E3n}\\
&w \in U \Rightarrow w_3\text{ is an infinite subset of }3\mathbb{N} + 2;\label{Ew3}\\
&w \in U \Rightarrow (3\mathbb{N} + 1) \subseteq w_1 \cup
w_2;\label{Ew1w2}\\
&w \in U \Rightarrow w_2 = Cl^-(w_3);\label{Ew2w3}\\
&w \in U \Rightarrow w_2 \cap \mathbf{v}_2 \neq
\emptyset.\label{Ew2v2}
\end{align}
\end{lemma}
\begin{proof}
Again, most of the statements are in fact self-evident given the
above definitions. As an example we sketch proofs of the following
ones:

\eqref{E3n}  We have $3\mathbb{N} \subseteq \mathbf{v}_1$, since
$Cl(3\mathbb{N} + 2) = (3\mathbb{N} + 2) \cup Cl^-(3\mathbb{N} +
2) \subseteq (3\mathbb{N} + 2) \cup (3\mathbb{N} + 1)$ by
\eqref{Eclo-}. Therefore, $Cl(3\mathbb{N}) \subseteq
Cl(\mathbf{v}_1) = \mathbf{v}_1$. We also have $\mathbb{N}_0
\subseteq \mathbf{v}_1$ by \eqref{En0}.

\eqref{Ew1w2} Since $w \in U$, by $\mathbf{v} \trianglelefteq w$
we have $w_3 \subseteq \mathbf{v}_3$, therefore $\mathbf{v}_1 \cup
\mathbf{v}_2 \subseteq w_1 \cup w_2$. And we clearly have
$(3\mathbb{N} + 1) \subseteq \mathbb{N} \smallsetminus
(3\mathbb{N} + 2) = \mathbf{v}_1 \cup \mathbf{v}_2$.

\eqref{Ew2w3} If $x \in w_3$, then by \eqref{Ew3} $x \in
3\mathbb{N} + 2$, therefore $\tilde{x} \in 3\mathbb{N} + 1$ whence
by \eqref{Ew1w2} $\tilde{x} \in w_1 \cup w_2$, so $\tilde{x}
\notin w_3$. If $\tilde{x} \in w_1$ then by closure of $w_1$ we
must have $x \in w_1$, a contradiction. Therefore,  $\tilde{x} \in
w_2$. In the other direction, if $x \in w_2 \subseteq
\mathbf{v}_2$ (by \eqref{Ev2}), then $x \in Cl^-(3\mathbb{N} +
2)$, therefore, $\tilde{x} \in (3\mathbb{N} + 2)$. So, $\tilde{x}
\notin w_2$, but then $\tilde{x} \in w_1 \cup w_3$. If $\tilde{x}
\in w_1$, then by closure $x \in w_1$, a contradiction. Therefore,
$\tilde{x} \in w_3$.

\eqref{Ew2v2} By \eqref{Ew3}, $w_3$ is non-empty and $w_3
\subseteq 3 \mathbb{N} + 2$. So, choose $k$ such that $n = 3k + 2
\in w_3$ and consider $\tilde{n} \in 3\mathbb{N} + 1$. By
\eqref{Ew2w3}, $\tilde{n} \in Cl^-(w_3) = w_2$, and given
\eqref{Ev2}, $\tilde{n} \in B \cap \mathbf{v}_2$.
\end{proof}

 ${\cal M}_1$ and
${\cal M}_2$ are of the form ${\cal M} = \langle W, \leq,
\mathbf{w}, D, \phi \rangle$ they are defined as follows.

\begin{enumerate}
\item The base points for ${\cal M}_1$ and ${\cal M}_2$ are
$\mathbf{v}$ and $\mathbf{u}$ respectively; \item The sets of
states for the models are as follows.
\begin{enumerate}
\item $W_1 = U$; \item $W_2 = U \cup \{\mathbf{u}\}$;
\end{enumerate}
\item The ordering on both ${\cal M}_1$ and ${\cal M}_2$ is
$\trianglelefteq$; \item $D_1 = D_2 = \mathbb{N}$; \item For $i =
1,2$, the values $\phi_i$ assigned to  the predicate symbols $P$,
 $Q$, $R$ and the propositional letter $s$ are defined by
\[ \phi_i(P) = \{ \langle v,a \rangle | a \in v_1 \cup v_2 \}, \]
\[ \phi_i(Q) = \{ \langle v,a \rangle | a \in v_1 \}, \]
\[ \phi_i(R) = \{ \langle v,a \rangle | \gamma(a) \in v_1 \}, \]
\[ \phi_i(s) = U. \]
\end{enumerate}

For $i = 1,2$ we denote extension of a predicate letter $\Pi \in
\{ P, Q, R \}$ at $w \in W_i$, that is  to say, the right
projection of the set $ \phi_i(\Pi) \cap (\{ w \}
\times\mathbb{N})$, by $\phi^w_i(\Pi)$.

Now, consider theory $T$:

\begin{align}
&\forall x(s \to \exists y(P(y) \wedge (Q(y) \to
R(x))))\label{E1}\\
&\neg\forall xR(x)\label{E2}\\
&\forall x(P(x) \to (Q(x) \vee s))\label{E3}
\end{align}

\begin{lemma}\label{Lsat}
Both ${\cal M}_1$ and ${\cal M}_2$ satisfy $T$.
\end{lemma}
\begin{proof}
It is sufficient to consider ${\cal M}_2$ only, since ${\cal M}_1$
is an `accessibility-closed' submodel of ${\cal M}_2$.

${\cal M}_2$ clearly satisfies \eqref{E3} since $s$ is universally
true in every world of this model except for $\mathbf{u}$. But for
this world we also have $\phi^{\mathbf{u}}_2(P) \subseteq
\phi^{\mathbf{u}}_2(Q)$, so \eqref{E3} holds at ${\cal M}_2$.

Choose any $w \in W_2$. If $w \notin U$, then $w = \mathbf{u}$ and
$w_1 \cap \mathbf{v}_2 = \emptyset$. Therefore, it follows from
\eqref{Ew2v2} that for any $w \in W_2$ we can choose an $n$ such
that $3n + 1 \in \mathbf{v}_2 \smallsetminus w_1$. For this $n$ we
will have $3n + 1 \notin \phi^{w}_2(Q)$, and, therefore $n \notin
\phi^{w}_2(R)$. Hence \eqref{E2} holds as well.

Finally, consider \eqref{E1}. If for $w \in W_2$ we have $w
\Vdash_2 s$ then $w \in U$, therefore, reasoning exactly as in
\cite{MOU2012}(i.e. choosing $\gamma(a)$ for every $a \in
\mathbb{N}$) one can show that the consequent of \eqref{E1} is
true at $w$.
\end{proof}

\begin{lemma}\label{Limpl}
$T$ implicitly defines $s$.
\end{lemma}
\begin{proof}
Let ${\cal M} = \langle W, \leq, \mathbf{w}, D, \phi \rangle$ and
let $w \in W$. If $w \Vdash_{{\cal M}} s$ then choose $a \notin
\phi^{w}(R)$ (such an $a$ exists since ${\cal M} \models
\eqref{E2}$) and consider $b$ such that $w \Vdash_{{\cal M}}  P(b)
\wedge (Q(b) \to R(a))$. For this $b$ we will have $b \in
\phi^{w}_{{\cal M}}(P) \smallsetminus \phi^{w}_{{\cal M}}(Q)$.

In the other direction, let for some $b$ it is true that $b \in
\phi^{w}_{{\cal M}}(P) \smallsetminus \phi^{w}_{{\cal M}}(Q)$.
Then, by ${\cal M} \models \eqref{E3}$, we will have $w
\Vdash_{{\cal M}} s$.
\end{proof}

\begin{definition}\label{DefinitionofZ}
Relative to the models ${\cal M}_1$ and ${\cal M}_2$, the relation
$Z$ is defined as follows:
\begin{enumerate}
\item $Z \subseteq \bigcup_{k \geq 0} [ (W_1 \times D_1^{k})
  \times (W_2 \times D_2^{k})] \cup [ (W_2 \times D_2^{k})
  \times (W_1 \times D_1^{k})]$;
\item $\langle (A,B,C), \vec{d} \rangle \: Z \: \langle (D,E,F),
\vec{e} \rangle$,
  where $\vec{d} \in D_i^{k}$, $\vec{e} \in D_j^{k}$, $\{i,j\} = \{1,2\}$,
  if and only if the following conditions hold:
\begin{enumerate}
\item Relation $[\vec{d} \mapsto \vec{e}]$ is a bijection;

\item  If $1 \leq l \leq k$, then  $d_l \in (3\mathbb{N} \cup
3\mathbb{N} + 2) \Leftrightarrow e_l \in (3\mathbb{N} \cup
3\mathbb{N} + 2)$;

\item  If $1 \leq l \leq k$, then  $d_l \in \mathbb{N}_0
\Leftrightarrow e_l \in \mathbb{N}_0$;

\item If $1 \leq l,m \leq k$ then $d_l = \widetilde{d_m}
\Leftrightarrow e_l = \widetilde{e_m}$;

\item If $1 \leq l \leq k$ and  $d_l \in A$ then $e_l \in D$;

\item If $1 \leq l \leq k$ and $d_l \in B$, then $e_l \in D \cup
E$.
\end{enumerate}
\end{enumerate}
\end{definition}

It is easy to see that in the case when  both $\langle (A,B,C),
\vec{d} \rangle \: Z \: \langle (D,E,F), \vec{e} \rangle$ and
 $\langle (D,E,F), \vec{e} \rangle \: Z \: \langle (A,B,C), \vec{d}
\rangle$ hold, conditions 2(e),(f) of this definition are
equivalent, modulo other restrictions,  to the following ones:
\[
d_l\in A \text{ iff } e_l\in D;
\]
\[
d_l\in B \text{ iff } e_l\in E,
\]

for every $1 \leq l \leq k$.

Further, it follows from conditions 2(b),(c) and \eqref{En0} that
\[
d_l \in Cl(3\mathbb{N} \cup 3\mathbb{N} + 2) \text{ iff }  e_l \in
Cl(3\mathbb{N} \cup 3\mathbb{N} + 2);
\]
\[
d_l \in Cl^-(3\mathbb{N} \cup 3\mathbb{N} + 2) \text{ iff }  e_l
\in Cl^-(3\mathbb{N} \cup 3\mathbb{N} + 2),
\]
for every $1 \leq l \leq k$.

\begin{lemma}\label{Zanasimulation}
 The relation $Z$ in Definition \ref{DefinitionofZ} is a CD-asimulation between the $L(P,Q,R)$-reducts of G-models ${\cal M}_1$ and ${\cal M}_2$.
\end{lemma}

\begin{proof}
The first condition in definition of asimulation is true by
definition. For the second condition, assume that $v, \vec{d} Z w,
\vec{e}$ and that $v \Vdash_i P[\vec{\mathbf d}]$, where $P(x_l)$
is atomic, $\vec{d} \in D_i^{k}$, and $1 \leq l \leq k$. Thus we
have $v \Vdash_i P[{\mathbf d}]$, where $d = d_l$, so that $d \in
v_1 \cup v_2$; it follows that $e = e_l \in w_1 \cup w_2$, by
Definition \ref{DefinitionofZ}, showing that $v \Vdash_j
P[\vec{\mathbf e}]$. The proof for atomic formulas $Q(x_l)$ is
similar. Finally, assume that $v \Vdash_i R[\vec{\mathbf d}]$,
where $R(x_l)$ is atomic, $\vec{d} \in D_i^{k}$, and $1 \leq l
\leq k$. Thus we have $v \Vdash_i R[{\mathbf d}]$, where $d =
d_l$, so that $\gamma(d) \in v_1$; since $v_1$ is closed, we have
$d \in v_1$ but then by Definition \ref{DefinitionofZ} it follows
that $e = e_l \in w_1$, and, given that $w_1$ is closed as well,
we get $\gamma(e) \in w_1$ and $w \Vdash_j R[{\mathbf e}]$ for the
corresponding $j \in \{ 1,2 \}$.

For the third condition, assume that $t, \vec{d} Z u, \vec{e}$,
where $t = (A,B,C)$, $u = (D,E,F)$, and $u \leq_j v$, $v =
(G,H,I)$. By definition, $u \trianglelefteq v$. Two cases arise
here: $B$ is infinite, or $B = \emptyset$.

In the first case, we clearly have $(A,B,C) \in U$. We define $w =
(J,K,L)$ as follows:
\begin{eqnarray*}
 J & = & Cl((A \smallsetminus \vec{d}) \cup [\vec{d} \mapsto \vec{e}]^{-1}(G));\\
 K & = & ((B \smallsetminus \vec{d}) \cup [\vec{d} \mapsto
 \vec{e}]^{-1}(H))\smallsetminus J;\\
 L & = & ((C \smallsetminus \vec{d}) \cup [\vec{d} \mapsto \vec{e}]^{-1}(I))\smallsetminus J.
\end{eqnarray*}

\emph{Claim 1}: $(J,K,L) \in U$.

 Sets $(A \smallsetminus \vec{d}) \cup [\vec{d}
\mapsto \vec{e}]^{-1}(G)$, $(B \smallsetminus \vec{d}) \cup
[\vec{d} \mapsto
 \vec{e}]^{-1}(H)$ and $(C \smallsetminus \vec{d}) \cup [\vec{d} \mapsto
 \vec{e}]^{-1}(I)$ are clearly pairwise disjoint and their union is $\mathbb{N}$. But $J$,
 $K$, $L$ differ from these sets only in that all the elements of the set $Cl^-((A \smallsetminus \vec{d}) \cup [\vec{d}
\mapsto \vec{e}]^{-1}(G))$ were removed
 from the last two components of $(J,K,L)$ and moved to the first
 one. Therefore $J$,
 $K$, $L$ are pairwise disjoint, too, and we have  $J \cup K \cup L =
 \mathbb{N}$. Moreover,  set $3\mathbb{N}\smallsetminus\vec{d}$ is obviously infinite, and by \eqref{E3n} we have
 \[
 (3\mathbb{N}\smallsetminus\vec{d}) \subseteq (\mathbf{v}_1\smallsetminus\vec{d}) \subseteq (A\smallsetminus\vec{d}) \subseteq
 J.
 \]
 Therefore, $J$ is infinite and it is closed by definition.

 \emph{Subclaim 1.1}. $L$ and $K$ are infinite.

 Since $(A,B,C) \in U$, $C$ is, by \eqref{Ew3}, an infinite subset of $(3\mathbb{N} +
 2)$.
 But then, since $\vec{d}$ is finite, $(C \smallsetminus
 \vec{d}) $ is infinite, too. So, if $(C \smallsetminus
 \vec{d}) \smallsetminus J$ is finite, then $(C \smallsetminus
 \vec{d}) \cap J$ must be infinite. Given that
 \[
 (C \smallsetminus
 \vec{d}) \cap  Cl(A \smallsetminus
 \vec{d}) \subseteq  C \cap  Cl(A) = C \cap A = \emptyset,
 \]
we obtain that $(C \smallsetminus
 \vec{d}) \cap Cl([\vec{d} \mapsto
 \vec{e}]^{-1}(G))$ must be infinite. But $Cl([\vec{d} \mapsto
 \vec{e}]^{-1}(G))$ is a closure of a finite set and therefore by \eqref{Eclofin} is itself finite, a
 contradiction. Therefore,  $(C \smallsetminus
 \vec{d}) \smallsetminus J \subseteq L$ is infinite.

 Moreover, if $n \in C \smallsetminus
 \vec{d}$ then by \eqref{Ew2w3} $\tilde{n} \in B$. Since for all $m,n \in \mathbb{N}$ we have $m \neq n \Rightarrow \tilde{m} \neq \tilde{n}$,
  the set $\{ \tilde{n} \mid n \in C
 \smallsetminus \vec{d} \}  \smallsetminus \vec{d}$ is an infinite subset of $K$.

\emph{Subclaim 1.2}. $K \subseteq \mathbf{v}_2$.

By \eqref{Ev2} we have $H \cap \vec{e} \subseteq H \subseteq
\mathbf{v}_2$ and also $B \smallsetminus \vec{d} \subseteq B
\subseteq \mathbf{v}_2$. If for some $1 \leq l \leq k$ $e_l \in H$
and $d_l \notin \mathbf{v}_2$, then $d_l \notin B$, therefore,
$d_l \in A \cup C$. If $d_l \in A$, then $e_l \in D \subseteq G$
which contradicts the assumption that $e_l \in H$. If $d_l \in C$,
then by \eqref{Ew3} $d_l \in 3\mathbb{N} + 2$, therefore, by
condition 2(b) of Definition \ref{DefinitionofZ}, $e_l \in
3\mathbb{N} \cup (3\mathbb{N} + 2)$ and so $e_l \notin
\mathbf{v}_2$, again a contradiction. So we have $[\vec{d} \mapsto
 \vec{e}]^{-1}(H)) \subseteq \mathbf{v}_2$ and, in sum, $K \subseteq \mathbf{v}_2$.

This completes the proof of our Claim 1.

\emph{Claim 2}: $(A,B,C) \trianglelefteq (J,K,L)$.

If $a \in A$, and $a$ is in $\vec{d}$, say $a = d_l$, then $e_l
\in D$, so $e_l \in G$, from which it follows that $d_l = a \in
J$. So $A \subseteq J$. Since $I \subseteq F$, we have $[\vec{d}
\mapsto \vec{e}]^{-1}(I) \subseteq [\vec{d} \mapsto
\vec{e}]^{-1}(F) \subseteq C$, showing that $L \subseteq C$.

\emph{Claim 3}: For $1 \leq l \leq k$,
\[
d_l\in J \text{ iff } e_l\in G;
\]
\[
d_l\in K \text{ iff } e_l\in H.
\]

For the first part, if $d_l \in J$ then either $d_l \in
 Cl(A \smallsetminus
 \vec{d})$ or $d_l \in Cl([\vec{d} \mapsto \vec{e}]^{-1}(G))$.
 In the former case $d_l \in Cl(A) = A$, therefore $e_l \in D \subseteq G$.
 In the latter case either $d_l \in [\vec{d} \mapsto \vec{e}]^{-1}(G)$ and then clearly $e_l \in G$ or
 there exists $1 \leq m \leq k$ such that $d_l = \widetilde{d_m}$ and $d_m \in [\vec{d} \mapsto
 \vec{e}]^{-1}(G)$. But then we also have $e_m \in G$ and, by
condition 2(e) of Definition \ref{DefinitionofZ},
 $e_l = \widetilde{e_m}$, and, since $G$ is closed, we get $e_l \in G$.

 In the other direction, if  $e_l \in G$, then $d_l \in [\vec{d} \mapsto
 \vec{e}]^{-1}(G)$ whence $d_l \in J$.

 For the second part, if $d_l \in K$, then  $d_l \in [\vec{d} \mapsto
 \vec{e}]^{-1}(H)$ whence  $e_l \in H$.

 In the other direction, if $e_l \in H$, then clearly $d_l \in [\vec{d} \mapsto
 \vec{e}]^{-1}(H)$. If, moreover,  $d_l \in J$, then either $d_l \in
 Cl(A \smallsetminus
 \vec{d})$ or $d_l \in Cl([\vec{d} \mapsto \vec{e}]^{-1}(G))$. In
 the former case $d_l \in Cl(A) = A$, therefore, we must have $e_l
 \in G$ which contradicts the choice of $e_l$. In the latter case,
 since we clearly have that  $d_l \notin [\vec{d} \mapsto
 \vec{e}]^{-1}(G)$, there must be $1 \leq m \leq k$ such that $d_l = \widetilde{d_m}$ and $d_m \in [\vec{d} \mapsto
 \vec{e}]^{-1}(G)$. But then we also have $e_m \in G$ and, by
condition 2(e) of Definition \ref{DefinitionofZ},
 $e_l = \widetilde{e_m}$, and, since $G$ is closed, we get $e_l \in G$, which again contradicts the choice of $e_l$.
 Therefore, $d_l \in [\vec{d} \mapsto
 \vec{e}]^{-1}(H) \smallsetminus J \subseteq K$.

 In the second case we clearly have $(A, B, C) = \mathbf{u}$ and
 we define $w = (J,K,L)$ as follows:
\begin{eqnarray*}
 J & = & Cl((A \smallsetminus \vec{d}) \cup [\vec{d} \mapsto \vec{e}]^{-1}(G));\\
 K & = & ((Cl^-(3\mathbb{N} + 2) \smallsetminus \vec{d}) \cup [\vec{d} \mapsto
 \vec{e}]^{-1}(H))\smallsetminus J;\\
 L & = & ((3\mathbb{N} + 2) \smallsetminus \vec{d}) \cup [\vec{d} \mapsto \vec{e}]^{-1}(I))\smallsetminus J.
\end{eqnarray*}
We need now to reinstate our previous claims for these new
definitions.

\emph{Claim 1}: $(J,K,L) \in U$.

Since $A = \mathbf{u}_1 = \mathbf{v}_1$, we clearly have $J \cup K
\cup L = \mathbb{N}$. Moreover, $J \cap K = J \cap L = \emptyset$
by definition, and since it is clear that $(Cl^-(3\mathbb{N} + 2)
\smallsetminus \vec{d}) \cup [\vec{d} \mapsto
 \vec{e}]^{-1}(H)$ and $((3\mathbb{N} + 2) \smallsetminus \vec{d}) \cup [\vec{d} \mapsto
 \vec{e}]^{-1}(I)$ are disjoint, we also get that $K$ and $L$ are
 disjoint. Further, $J$ is closed by definition and contains $A \smallsetminus
 \vec{d}$ as an infinite subset.

\emph{Subclaim 1.1}. $L$ is an infinite subset of $3\mathbb{N} +
2$, $K$ is infinite.

 Since $\vec{d}$ is finite, $((3\mathbb{N} + 2) \smallsetminus
 \vec{d})$ is infinite, too. So, if $((3\mathbb{N} + 2) \smallsetminus
 \vec{d}) \smallsetminus J$ is finite, then $((3\mathbb{N} + 2) \smallsetminus
 \vec{d}) \cap J$ must be infinite. Given that
 \[
 ((3\mathbb{N} + 2) \smallsetminus
 \vec{d}) \cap  Cl(A \smallsetminus
 \vec{d}) \subseteq  (3\mathbb{N} + 2) \cap  Cl(A) \subseteq C \cap A  = \emptyset,
 \]
 we obtain that $((3\mathbb{N} + 2) \smallsetminus
 \vec{d}) \cap Cl([\vec{d} \mapsto
 \vec{e}]^{-1}(G))$ must be infinite. But $Cl([\vec{d} \mapsto
 \vec{e}]^{-1}(G))$ is a closure of a finite set and therefore, by \eqref{Eclofin} is itself finite, a
 contradiction. Therefore, $((3\mathbb{N} + 2) \smallsetminus
 \vec{d}) \smallsetminus J \subseteq L$ is infinite.

 Moreover, it is easy to see that if $n \in (3\mathbb{N} + 2) \smallsetminus
 \vec{d}$ then $\tilde{n} \in Cl^-(3\mathbb{N} + 2)\smallsetminus
 \vec{d}$. Since for all $m,n \in \mathbb{N}$ we have $m \neq n \Rightarrow \tilde{m} \neq \tilde{n}$, the set $Cl^-(3\mathbb{N} + 2)\smallsetminus
 \vec{d}$ is an infinite subset of $K$.
 It remains to verify $[\vec{d} \mapsto \vec{e}]^{-1}(I) \subseteq
(3\mathbb{N} + 2)$, so that we can be sure that $L \subseteq
3\mathbb{N} + 2$. Since $(A, B, C) = \mathbf{u}$, we know that
$(G,H,I) \in U$ and therefore $I \subseteq (3\mathbb{N} + 2)$. So,
by condition  2(b) of Definition \ref{DefinitionofZ}, if for $1
\leq l \leq k$ we have $d_l \in [\vec{d} \mapsto
\vec{e}]^{-1}(I)$, then $e_l \in I \subseteq (3\mathbb{N} + 2)$
and $d_l \in 3\mathbb{N} \cup (3\mathbb{N} + 2)$. But if $d_l \in
3\mathbb{N}$, then, by \eqref{E3n}, $d_l \in A$, therefore $e_l
\in G$, a contradiction. So we must have $d_l \in (3\mathbb{N} +
2)$

\emph{Subclaim 1.2}. $K \subseteq \mathbf{v}_2$.

By \eqref{Ev2} we have $H \cap \vec{e} \subseteq H \subseteq
\mathbf{v}_2$ and also $Cl^-(3\mathbb{N} + 2) \smallsetminus
\vec{d} \subseteq \mathbf{v}_2$. If, for some $1 \leq l \leq k$,
$e_l \in H$ and $d_l \notin \mathbf{v}_2$, then $d_l \in A \cup
(3\mathbb{N} + 2)$. If $d_l \in A$, then $e_l \in D \subseteq G$
which contradicts the assumption that $e_l \in H$. If $d_l \in
3\mathbb{N} + 2$, then, by condition  2(b) of Definition
\ref{DefinitionofZ}, $e_l \in 3\mathbb{N} \cup (3\mathbb{N} + 2)$
and so $e_l \notin \mathbf{v}_2$, which is a contradiction by
\eqref{Ev2}. So we have $[\vec{d} \mapsto
 \vec{e}]^{-1}(H) \subseteq \mathbf{v}_2$ and, in sum, $K \subseteq \mathbf{v}_2$.

This completes the proof of our Claim 1.

\emph{Claim 2}: $(A,B,C) \trianglelefteq (J,K,L)$.

$A \subseteq J$ can be verified as in Case 1 and $L \subseteq C$
follows from Claim 1.1 and the fact that $(A, B, C) = \mathbf{u}$.

\emph{Claim 3}: For $1 \leq l \leq k$,
\[
d_l\in J \text{ iff } e_l\in G;
\]
\[
d_l\in K \text{ iff } e_l\in H.
\]

This Claim, again, can be verified as in Case 1. So, we have
completed the proof for the third condition.

For the fourth condition in definition of asimulation, assume that
$t \in W_i$, $(t,\vec{d} Z u, \vec{e})$, where $t = (G,H,I)$, $u =
(J,K,L)$, and $f \in D_i$. Three cases are possible:

\emph{Case 1}. $f \notin Cl(\vec{d})$. Then three subcases are
possible:

\emph{Subcase 1.1} $f \in  \mathbb{N}_0$. Then choose any $g \in
\mathbb{N}_0 \smallsetminus Cl(\vec{e})$. This is possible since
this set is infinite. By \eqref{E3n} we have $\mathbb{N}_0
\subseteq \mathbf{v}_1 \subseteq u_1$ and so we are done.

\emph{Subcase 1.2} $f \in 3\mathbb{N} \cup 3\mathbb{N} + 2$. Then
choose any $g \in 3\mathbb{N} \smallsetminus Cl(\vec{e})$. This is
possible since this set is infinite. By \eqref{E3n} we have
$3\mathbb{N} \subseteq \mathbf{v}_1 \subseteq u_1$ and so we are
done.

\emph{Subcase 1.3.} $f \in Cl^-(3\mathbb{N} \cup 3\mathbb{N} +
2)$. Then choose any $g \in Cl^-(3\mathbb{N})$. This is possible
since this set is infinite. By \eqref{E3n} we have
$Cl^-(3\mathbb{N}) \subseteq \mathbf{v}_1 \subseteq u_1$ and so we
are done.

\emph{Case 2}. $f \in \vec{d}$. Then set $g: = e_l$.

\emph{Case 3}. $f \in Cl^-(\vec{d})$. This means that for some $1
\leq l \leq k$ $f = \widetilde{d_l} \neq d_l$. Therefore, by
conditions 2(b)--(d) of Definition \ref{DefinitionofZ}, $f, d_l,
e_l, \widetilde{e_l} \in Cl(3\mathbb{N} \cup 3\mathbb{N} + 2)$,
and $\widetilde{e_l} \neq e_l$. So, in particular, by \eqref{En0},
we have $f, \widetilde{e_l} \notin \mathbb{N}_0$. Then set $g :=
\widetilde{e_l}$. Let us make sure that conditions of Definition
\ref{DefinitionofZ} are satisfied. For example, by \eqref{Eiff1},
\eqref{Eiff2} we have:
\begin{align*}
f = \widetilde{d_l} \in 3\mathbb{N} \cup 3\mathbb{N} + 2
&\Leftrightarrow d_l \in Cl^-(3\mathbb{N} \cup 3\mathbb{N} +
2)\\
&\Leftrightarrow e_l \in Cl^-(3\mathbb{N} \cup 3\mathbb{N} +
2)\\
&\Leftrightarrow g = \widetilde{e_l} \in 3\mathbb{N} \cup
3\mathbb{N} + 2.
\end{align*}
 Further, if $f =
\widetilde{d_l} \in G$ then, since $G$ is closed, $d_l \in G$, but
then also $e_l \in J$ and, by closure of $J$, $g = \widetilde{e_l}
\in J$. Moreover, if $f = \widetilde{d_l} \in H$, then $H \neq
\emptyset$, so $t \in U$ and, by \eqref{Ew3}, $d_l \in I \subseteq
3\mathbb{N} + 2$. But then, by condition 2(b) of Definition
\ref{DefinitionofZ}, $e_l \in 3\mathbb{N} \cup 3\mathbb{N} + 2
\subseteq J \cup L$, where the latter inclusion holds due to
\eqref{Ev2}--\eqref{Ew3}. Therefore, $g = \widetilde{e_l} \in
Cl^-(J \cup L) \subseteq K$ by \eqref{Ew2w3} and since by closure
of $J$ we have $Cl^-(J) = \emptyset$.

The fifth condition is proved by a similar argument. Namely,
assume that $t \in W_i$, $(t,\vec{d} Z u, \vec{e})$, where $t =
(G,H,I)$, $u = (J,K,L)$, and $g \in D_j$.Three cases are possible:

\emph{Case 1}. $g \notin Cl(\vec{e})$. Then three subcases are
possible:

\emph{Subcase 1.1} $g \in \mathbb{N}_0$. Then choose any $f \in
\mathbb{N}_0 \smallsetminus Cl(\vec{d})$. This is possible since
this set is infinite. By \eqref{E3n} we have $\mathbb{N}_0
\subseteq \mathbf{v}_1 \subseteq u_1$ and so we are done.

\emph{Subcase 1.2} $g \in 3\mathbb{N} \cup 3\mathbb{N} + 2$. Then
choose any $f \in (I \cap  (3\mathbb{N} + 2)) \smallsetminus
Cl(\vec{d})$. This is possible since it follows from
\eqref{Eclofin}, \eqref{Ew3} that this set is infinite.

\emph{Subcase 1.3.} $g \in Cl^-(3\mathbb{N} \cup 3\mathbb{N} +
2)$. Then $g \in J \cup K$. Choose any $f \in Cl^-(I \cap
(3\mathbb{N} + 2)) \smallsetminus Cl(\vec{d})$. This is possible
since this set is infinite. Indeed, $I \cap (3\mathbb{N} + 2)$ is
infinite by \eqref{Ew3}, but then, by \eqref{Eclofin}, $Cl^-(I
\cap (3\mathbb{N} + 2))$ is infinite and $Cl(\vec{d})$ is finite.
By \eqref{Ew2w3} we have $f \in H$ and so we are done.

\emph{Case 2}. $g \in \vec{e}$. Then set $f: = d_l$.

\emph{Case 3}. $g \in Cl(\vec{e}) \smallsetminus \vec{e}$. This
means that for some $1 \leq l \leq k$ $g = \widetilde{e_l} \neq
e_l$. Therefore, by conditions 2(b)--(d) of Definition
\ref{DefinitionofZ}, $d_l, \widetilde{d_l}, e_l, g \in
Cl(3\mathbb{N} \cup 3\mathbb{N} + 2)$, and $\widetilde{d_l} \neq
d_l$. So in particular, by \eqref{En0}, $\widetilde{d_l}, g \notin
\mathbb{N}_0$. Then set $f := \widetilde{d_l}$. Let us make sure
that conditions of Definition \ref{DefinitionofZ} are satisfied.
For example, by \eqref{Eiff1}, \eqref{Eiff2} we have:
\begin{align*}
g = \widetilde{e_l} \in 3\mathbb{N} \cup 3\mathbb{N} + 2
&\Leftrightarrow e_l \in Cl^-(3\mathbb{N} \cup 3\mathbb{N} +
2)\\
&\Leftrightarrow d_l \in Cl^-(3\mathbb{N} \cup 3\mathbb{N} + 2)\\
&\Leftrightarrow f = \widetilde{d_l}\in 3\mathbb{N} \cup
3\mathbb{N} + 2.
\end{align*}
Further, if $f = \widetilde{d_l} \in G$ then, since $G$ is closed,
$d_l \in G$, but then also $e_l \in J$ and, by closure of $J$, $g
= \widetilde{e_l} \in J$. Moreover, if $f = \widetilde{d_l} \in
H$, then $H \neq \emptyset$, so $t \in U$ and, by \eqref{Ew3},
$d_l \in I \subseteq 3\mathbb{N} + 2$. But then, by condition 2(b)
of Definition \ref{DefinitionofZ}, $e_l \in 3\mathbb{N} \cup
3\mathbb{N} + 2 \subseteq J \cup L$, where the latter inclusion
holds due to \eqref{Ev2}--\eqref{Ew3}. Therefore, $g =
\widetilde{e_l} \in Cl^-(J \cup L) \subseteq K$ by \eqref{Ew2w3}
and since by closure of $J$ we have $Cl^-(J) = \emptyset$.
\end{proof}
\begin{theorem}
Intuitionistic predicate logic of constant domains does not have
Beth definability property.
\end{theorem}
\begin{proof}
Consider Theory $T$. According to Lemma \ref{Limpl}, it implicitly
defines $s$. If the logic in question enjoys Beth definability
property, then there is a sentence $\Theta \in L(P,Q,R)$ such that
$T \models s \leftrightarrow \Theta$. But then, by Lemma
\ref{Lsat}, we must have $\mathbf{v} \Vdash_1 \Theta$. It follows
from Definition \ref{DefinitionofZ} that $\langle\mathbf{v},
\Lambda\rangle Z \langle\mathbf{u}, \Lambda\rangle$, where
$\Lambda$ is the empty sequence of objects. Therefore, by Lemma
\ref{Zanasimulation}, we must have  $\mathbf{u} \Vdash_2 \Theta$.
But, given that $s$ is not true at $\mathbf{u}$, this gives us a
contradiction with Lemma \ref{Lsat}.
\end{proof}

\bibliographystyle{habbrv}

\bibliography{logic}

}
\end{document}